\documentclass[a4paper,12pt]{amsart}
\usepackage{amsfonts}
\usepackage{amsmath,amssymb,latexsym,amsfonts,amscd}

\title{On pluricanonical systems of algebraic varieties of general type}
\author{Meng Chen}
\address{\rm Institute of Mathematics, Fudan University,
Shanghai, 200433, China} \email{mchen@fudan.edu.cn}
\thanks{Supported by National Outstanding
Young Scientist Foundation (\#10625103) and NNSFC Key project
(\#10731030)}

\newcommand{\bQ}{{\mathbb Q}}
\newcommand{\bP}{{\mathbb P}}
\newcommand{\roundup}[1]{\lceil{#1}\rceil}

\newcommand\lra{\longrightarrow}

\newcommand\OO{{\mathcal{O}}}

\newtheorem{thm}{Theorem}[section]
\newtheorem{lem}[thm]{Lemma}
\newtheorem{cor}[thm]{Corollary}
\newtheorem{prop}[thm]{Proposition}

\theoremstyle{definition}
\newtheorem{defn}[thm]{Definition}
\newtheorem{setup}[thm]{}

\newtheorem{rem}[thm]{Remark}
\newtheorem{OP}[thm]{Problem}
\theoremstyle{remark}

\begin{document}
\begin{abstract} We extend Koll\'ar's technique to
look for an explicit function $h(n)$ with $\varphi_m$ birational
onto its image for all integers $m\geq h(n)$ and for all
$n$-dimensional nonsingular projective varieties of general type.
\end{abstract}
\maketitle
\pagestyle{myheadings} \markboth{\hfill M. Chen\hfill}{\hfill
Pluricanonical systems\hfill}

\section{\bf Introduction}

One of the fundamental problems in birational geometry is to find a
constant $r_n>0$ such that the $r_n$-canonical map is an Iitaka
fibration for any $n$-dimensional projective variety with positive
Kodaira dimension. It is well-known that one may take $r_2=5$ (see
Bombieri \cite{Bom}) for surfaces of general type. The existence of
$r_n$ for varieties of general type was proved by
Hacon-M$^\text{c}$kernan \cite{H-M}, Takayama \cite{Taka} and Tsuji
\cite{Tsuji} and, recently, $r_3\leq 73$ was proved by Chen-Chen
\cite[Theorem 1.1]{Explicit_II}. Some other relevant results with
regard to the existence of $r_n$ have been already proved by
Chen-Hacon \cite{C-H}, Pacienza \cite{Pa} and Viehweg-Zhang
\cite{Viehweg-Zhang}. The following problem, however, is still open:

\begin{OP}\label{OP} To find an explicit constant $\mu_n$ ($n\geq 4$)
such that the m-canonical map $\varphi_m$ is birational onto its
image for all $m\geq \mu_n$ and for all $n$-dimensional projective
varieties of general type.
\end{OP}

Let $V$ be a $n$-dimensional nonsingular projective variety of
general type. Denote by $K_V$ a canonical divisor on $V$. A
reasonable strategy for studying Problem \ref{OP} is composed of two
steps:
\begin{itemize}
\item[[a]] To find a positive integer $m_0$ such that
$h^0(V, m_0K_V)\geq 2$;

\item[[b]] To find an explicit function $g(m_0,n)$ such
that the $m$-canonical map $\varphi_m$ is birational for all $m\geq
g(m_0,n)$.
\end{itemize}

This strategy works well in dimension 3 (see, for instance,
\cite{Explicit_I, Explicit_II}). In fact, since Reid \cite{YPG} has found the Riemann-Roch formula for minimal 3-folds, we \cite[Theorem 1.1]{Explicit_I} managed to prove $h^0(V, m_0K_V)\geq 2$ by utilizing that formula. On the other hand, Koll\'ar \cite{Kol} and myself \cite{JPAA} have an effective formula in dimension 3 for Step (b).
Generally, for Step [a], since the
classification to 4-dimensional terminal singularities is still
incomplete, there is no known Riemann-Roch formula for $\chi(mK)$
on minimal varieties. Thus to compute
$P_m:=h^0(V, m_0K_V)$ is still expected.  Though Koll\'ar \cite{Kol}
has essentially solved Step [b], what we are more concerned here is the stable birationality, to be worked out by an improved and
generalized technique, of linear systems $|mK+\roundup{Q}|$ where
$Q$ is any nef $\bQ$-divisor. Hence this paper can be regarded as a
remark or complementary to Koll\'ar's method. We will build up some
new results about induced fibrations from pluricanonical systems.

Assume $P_{m_0}(V):=h^0(V, \OO_V(m_0K_V))\geq 2$ for some positive
integer $m_0$. Set $\varphi_{m_0}:=\Phi_{|m_0K_V|}$, the
$m_0$-canonical map of $V$. We define {\it the $m_0$-canonical
dimension} $\iota:=\dim \overline{\varphi_{m_0}(V)}$. Clearly $1\leq
\iota\leq n$. In order to formulate our statements, we introduce the
following:

\begin{defn} Define $\lambda(V)$ to be the smallest positive integer such that
$P_{\lambda(V)}(V)\geq 2$ for a given $n$-dimensional projective
variety $V$ of general type. Define
$\lambda_n:=\text{sup}\{\lambda(V)| \dim(V)=n\}$. 
\end{defn}

\begin{rem} According to Hacon-M$^\text{c}$kernan \cite{H-M}, Takayama \cite{Taka} and Tsuji
\cite{Tsuji}, one knows
$\lambda_n<+\infty$. Therefore an assumption like
$P_{m_0}\geq 2$ is reasonable and natural.
\end{rem}

The main result of this paper is the following which, at least,
induces new results for the case $n=4$:

\begin{thm}\label{main}
Let $V$ be a $n$-dimensional ($n\geq 3$) nonsingular projective
variety of general type. Let $Q$ be a nef $\bQ$-divisor on $V$.
Assume $P_{m_0}\geq 2$ for some positive integer $m_0$.  Then the
linear system $|mK_V+\roundup{Q}|$ defines a birational map onto its
image for all integers $m\geq \varepsilon(\iota)$ where
$\varepsilon(\iota)$ is a function as follows:
\begin{itemize}
\item[(1)] when $\iota\geq n-2$, $\varepsilon(\iota)=\text{min}\{4m_0+4,
57\}\cdot(2m_0+1)^{n-3}+m_0(n-2)+2$;

\item[(2)] when $\iota=n-3$, $\varepsilon(\iota)=75(2m_0+1)^{n-3}+m_0(n-3)+2$;

\item[(3)] when $\iota\leq n-4$, $\varepsilon(\iota)=(2m_0+1)^{\iota-1}
w_{n-\iota+1}+m_0(\iota-1)+2$ where $w_{n-\iota+1}$ can be obtained
by the number sequence $\{w_t\}_{t=4}^{n-\iota+1}$ with
$w_i=\widetilde{\lambda}_i+w_{i-1}(2\widetilde{\lambda}_i+1)$,
$w_4=151\widetilde{\lambda}_4+75$,
$\widetilde{\lambda}_{n-\iota+1}=m_0$ and, for all other $i$,
$\widetilde{\lambda}_i=\lambda_i$.
\end{itemize}
By taking $Q=0$, $\varphi_m$ is birational for all integers $m\geq
\varepsilon(\iota)$.
\end{thm}

In particular we have the following:

\begin{cor}\label{general} Let $V$ be a nonsingular
projective n-dimensional ($n\geq 4$) variety of general type. Then
$\varphi_m$ is birational for all integers $m\geq w_n+2$ where $w_n$
is obtained by the number sequence $\{w_t\}_{t=4}^n$ with
$w_i=\lambda_i+w_{i-1}(2\lambda_i+1)$ and $w_4=151\lambda_4+75$.
\end{cor}

\begin{cor}\label{4-fold} (= Corollary \ref{4}) Let $V$ be a nonsingular
projective 4-fold of general type with $P_{m_0}\geq 2$ for some
integer $m_0>0$. Let $Q$ be any nef $\bQ$-divisor on $V$. Then
$\Phi_{|mK_V+\roundup{Q}|}$ (in particular, $\varphi_m$) is
birational for all $m\geq 151m_0+77.$
\end{cor}

Theorem \ref{main} also implies the following:

\begin{cor} An explicit constant $\mu_n$ mentioned in Problem
\ref{OP} can be found by means of Theorem \ref{main} if and only if explicit constants $\rho_k$ for all $k\leq n$
can be found such that the pluri-genus $P_{\rho_k}\geq 2$ for all
$k$-dimensional projective varieties of general type.
\end{cor}

{}From this point of view, a Riemann-Roch formula for
$\chi(\OO(mK))$ is of key importance just like what Reid has done
in \cite[last section]{YPG} for threefolds.

This paper is organized as follows. First we fix the notation for
the map $\varphi_{m_0}$. Then we systematically study the property,
of the induced fibration $f\colon X'\lra B$, which generalizes known
inequalities in 3-dimensional case. In Section 3, we will improve
known results on surfaces and 3-folds. Theorems in Section 4 are
original. We will prove the main theorem by induction in the last
section.

We always use the symbol $\equiv$ to denote numerical equivalence while $\sim$
means linear equivalence.

\section{\bf Properties of canonically induced fibrations}

In this paper, $V$ is always a $n$-dimensional nonsingular
projective variety of general type. Let $m_0$ be a positive integer.
Assume that $\Lambda\subset |m_0K_V|$ is a sub-linear system such
that $\Phi_{\Lambda}(V)=1$ where $\Phi_{\Lambda}$ is the rational
map defined by $\Lambda$. We call $\Lambda$ a {\em pencil} contained
in $|m_0K_V|$. Note that such a pencil always exists. For instance,
a 2-dimensional sub-space of $H^0(V, m_0K_V)$ corresponds to a
pencil. We will study properties of the rational map
$\Phi_{\Lambda}$ in this section.

\begin{setup}{\bf Existence of minimal models.}  By recent works of
Birkar-Cascini-Hacon-M$^\text{c}$kernan \cite{BCHM},
Hacon-M$^\text{c}$kernan\cite{HM2} and Siu \cite{Siu}, $V$ has a
minimal model. Again by \cite{HM2}, any fibration $f\colon Y\lra B$,
from a nonsingular variety $Y$ of general type to a smooth curve
$B$, has a relative minimal model. According to the established
Minimal Model Program (MMP), one may always assume that a minimal
model has at worst $\bQ$-factorial terminal singularities.

{}From this point of view, it suffices to study $\varphi_m$ on
minimal models.
\medskip

{\bf ($\ddag$) Throughout $X$ always denotes a minimal model of
$V$.}
\medskip

Because $\Phi_{\Lambda}$ can be defined on a Zariski open subset of
$X$, we may also regard $\Lambda$ as a {\em pencil} on the minimal
model $X$.
\end{setup}

\begin{setup}\label{setup}{\bf Set up for $\Phi_{\Lambda}$.} Denote
by $\mu\colon V\dashrightarrow X$ the birational contraction map.
Because $P_{m_0}(X)=P_{m_0}(V)\geq 2$, we may fix an effective Weil
divisor $K_{m_0}\sim m_0K_X$ on $X$ and a divisor
$\widetilde{K}_{m_0}\sim m_0K_V$ on $V$. Take successive blow-ups
$\pi\colon X'\rightarrow X$ along nonsingular centers, such that the
following conditions are satisfied:
\medskip
\begin{itemize}
\item[(i)] $X'$ is smooth;

\item[(ii)] there is a birational morphism $\pi_{V}\colon X'\rightarrow V$
such that $\mu\circ \pi_V=\pi$;

\item[(iii)] the movable part $M_{0}$ of $\pi^*_{V}(\Lambda)$ is base
point free and so that $g:=\Phi_{\Lambda}\circ\pi_{V}$ is a
non-constant morphism;

\item[(iv)] $\pi^*(K_{m_0})
\cup\pi^*_{V}(\widetilde{K}_{m_0})$ has simple normal crossing
supports;

\item[(v)] for certain purpose $\pi$ even satisfies a couple of extra
conditions by further modifying $X'$. (This condition will be
specified in explicit whenever we need it.)
\end{itemize}
\medskip

We have a morphism $g\colon X'\longrightarrow W'\subseteq{\mathbb
P}^{N}$. Let $X'\overset{f}\longrightarrow
B\overset{s}\longrightarrow W'$ be the Stein factorization of $g$.
We have the following commutative diagram:\medskip

\begin{picture}(50,80) \put(100,0){$V$} \put(100,60){$X'$}
\put(170,0){$W'$} \put(170,60){$B$} \put(112,65){\vector(1,0){53}}
\put(106,55){\vector(0,-1){41}} \put(175,55){\vector(0,-1){43}}
\put(114,58){\vector(1,-1){49}} \multiput(112,2.6)(5,0){11}{-}
\put(162,5){\vector(1,0){4}} \put(133,70){$f$} \put(180,30){$s$}
\put(92,30){$\pi_{V}$}
\put(135,-8){$\varphi_{\Lambda}$}\put(136,40){$g$}
\end{picture}
\bigskip

Denote by $r(X)$ the Cartier index of $X$. We can write
$r(X)K_{X'}=\pi^*(r(X)K_X)+E_{\pi}$ where $E_{\pi}$ is a sum of
exceptional divisors. Recall that
$$\pi^*(K_X):=K_{X'}-\frac{1}{r(X)}E_{\pi}.$$ Clearly, whenever we
take the round-up of $m\pi^*(K_X)$ for $m>0$, we always have
$\roundup{m\pi^*(K_X)}\leq mK_{X'}.$

Denote by $M_{k, X'}$ the movable part of $|kK_{X'}|$ for any
positive integer $k>0$. We may write $m_0K_{X'}=_{\mathbb
Q}\pi^*(m_0K_X)+E_{\pi, m_0}=M_{m_0,X'}+Z_{m_0},$ where $Z_{m_0}$
is the fixed part of $|m_0K_{X'}|$ and $E_{\pi, m_0}$ an effective
${\mathbb Q}$-divisor which is a ${\mathbb Q}$-sum of distinct
exceptional divisors with regard to $\pi$.

Since $M_{0}\le M_{m_0,X'}\leq \pi^*(m_0K_X)$, we can write
$\pi^*(m_0K_X)=M_{0}+E_{\Lambda}'$ where $E_{\Lambda}'$ is an
effective $\mathbb{Q}$-divisor. By our assumption ($\Lambda$ is a
pencil), $B$ is a nonsingular complete curve. By Bertini's theorem,
a general fiber $F$ of the fibration $f\colon X'\lra B$ is a
$(n-1)$-dimensional nonsingular projective variety of general type.

Once a fibration $f\colon X'\lra B$ is obtained, we may take the
relative minimal model $f_0\colon X_0\lra B$ of $f$. Then we can
remodify $\pi$ and $\pi_V$ again such that a new birational model
$X''$ dominates $X'$ and $X_0$. Namely assume $\pi'\colon X''\lra
X'$ and $\pi''\colon X''\lra X_0$ are the birational morphisms, set
$f'':=f\circ \pi'$, then $f''=f_0\circ \pi''$.
\end{setup}

Therefore we can make the following:

\begin{setup}\label{aasum}{\bf Assumption on $X'$}. To avoid
too complicated notations, we may assume from the beginning that
$X'=X''$ by further birational modifications, i.e. there is a
contraction morphism $\theta:X'\lra X_0$ such that $f=f_0\circ
\theta$ and that $f_0$ is a relative minimal model of $f$.

With this assumption, we pick up a general fiber $F_0$ of $f_0$ and
set $\sigma:=\theta|_F$, then $\sigma\colon F\lra F_0$ is a
birational morphism onto the minimal model.
\end{setup}

Set $b:=g(B)$. We will study the geometry of $f$ according to the
value of $b$. In fact there are two cases:

\begin{itemize}
\item
[(i)] $b>0$, $M_0\sim\underset{i=1}{\overset{p}\sum} F_i\equiv pF$
where the $F_i$'s are different smooth fibers of $f$ for all $i$ and
$p\geq 2$;
\item
[(ii)] $b=0$, $M_0\sim pF\leq m_0\pi^*(K_X)$ with $p\geq 1$.
\end{itemize}

\begin{setup}\label{reduction}{\bf Reduction to problems on $X'$.} As we have seen, there is a
birational morphism $\pi_V\colon X'\lra V$. Let $m$ be a positive
integer and $Q$ a nef $\bQ$-divisor on $V$. Since
$${\pi_V}_*\OO_{X'}(mK_{X'}+\pi_V^*(\roundup{Q}))\cong\OO_V(mK_V+\roundup{Q})$$
and $mK_{X'}+\pi_V^*(\roundup{Q})\geq mK_{X'}+\roundup{Q'}$ where
$Q':=\pi_V^*(Q)$ is nef on $X'$, the birationality of
$\Phi_{|mK_{X'}+\roundup{Q'}|}$ implies that of
$\Phi_{|mK_V+\roundup{Q}|}$. Furthermore the fact $mK_{X'}\geq
m\pi^*(K_X)$ allows us to consider a smaller linear system on $X'$
like:
$$|K_{X'}+\roundup{(m-1)\pi^*(K_X)+Q'}|.$$
\end{setup}

The 3-dimensional version of the next lemma has appeared as
\cite[Lemma 3.4]{Chen-Zuo}.

\begin{lem}\label{b>0} Let $f_Y\colon Y\dashrightarrow B_0$ be
a rational map onto a smooth curve $B_0$ where $Y$ is a normal
projective minimal variety (i.e. $K_Y$ nef) with at worst terminal
singularities. Let $\pi_Y\colon Y'\rightarrow Y$ be any birational
modification from a nonsingular projective model $Y'$ such that
$g_Y:=f_Y\circ \pi_Y\colon Y'\longrightarrow B_0$ is a proper
morphism. Denote by $F_b$ any irreducible component in a general
fiber of $g_Y$. Assume $g(B_0)>0$. Then
$$\OO_{F_b}(\pi_Y^*(K_Y)|_{F_b})\cong \OO_{F_b}({\sigma'}^*(K_{F_{b,0}}))$$
where $F_{b,0}$ is a minimal model of $F_b$ and there is a
contradiction morphism $\sigma'\colon F_b\lra F_{b,0}$.
\end{lem}
\begin{proof} One has a morphism
$g_Y\colon Y'\longrightarrow B_0$.  By theorems of Shokurov
\cite{Sho} and Hacon-M$^\text{c}$kernan \cite{HsM}, each fiber of
$\pi_Y\colon Y'\longrightarrow Y$ is rationally chain connected.
Therefore, $g_Y(\pi_Y^{-1}(y))$ is a point for all $y\in Y$.
Considering the image $G\subset (Y \times B_0)$ of $Y'$ via the
morphism $(\pi_Y\times g_y)\circ \triangle_{Y'}$ where
$\triangle_{Y'}$ is the diagonal map $Y'\longrightarrow Y'\times
Y'$, one knows that $G$ is a projective variety. Let $g_1\colon
G\longrightarrow Y$ and $g_2\colon G\longrightarrow B_0$ be two
projection maps. Since $g_1$ is a projective morphism and even a
bijective map, $g_1$ must be both a finite morphism of degree 1 and
a birational morphism. Since $Y$ is normal, $g_1$ must be an
isomorphism. So $g_Y$ factors as $f_1 \circ \pi_Y$ where
$f_1:=g_2\circ g_1^{-1}\colon Y \rightarrow B_0$ is a well-defined
morphism. Let $Y\overset{f_0}\longrightarrow
B'\overset{s'}\longrightarrow B_0$ be the Stein factorization of
$f_1$. Set $f':=f_0\circ \pi_Y$. Then $F$ is a general fiber of
$f'$. Denote by $F_{b,0}$ a general fiber of $f_0$. Clearly
$K_{F_{b,0}}\sim {K_Y}|_{F_{b,0}}$ is nef and so $F_{b,0}$ is
minimal. So it is clear that $\pi_Y^*(K_Y)|_{F_b}\sim
\sigma'^*(K_{F_{b,0}})$ where we set $\sigma':= {\pi_Y}|_{F_b}\colon
F_b\lra F_{b,0}$.
\end{proof}

The above lemma clearly applies to our situation with $b>0$.
\medskip

Now we begin to study the case $b=0$. According to our definition,
$M_{m,F}$ denotes the movable part of $|mK_F|$ for any $m>0$.

Since $B\cong \bP^1$, one has $\OO_B(1)\hookrightarrow
f_*\omega_{X'}^{\otimes m_0}$. Then there is the inclusion:
$$f_*\omega_{X'/B}^{\otimes m}\hookrightarrow f_*\omega_{X'}^{\otimes
m(2m_0+1)}$$ for all integers $m>0$. Because
$f_*\omega_{X'/B}^{\otimes m}$ is semi-positive (see Viehweg
\cite{V}) and is thus a direct sum of line bundles of non-negative
degree, so it is generated by global sections. Therefore any local
section of $f_*\omega_{X'/B}^{\otimes m}$ can be extended to a
global one of $f_*\omega_{X'}^{\otimes m(2m_0+1)}$. This already
means
$$m(2m_0+1)\pi^*(K_X)|_F\geq M_{m(2m_0+1),X'}|_F\geq M_{m,F}.$$
Whenever $m$ is divisible by the Cartier index of $F_{0}$, the Base
Point Free Theorem says that $M_{m,F}\sim m\sigma^*(K_{F_0})$. Thus
we get
$$\pi^*(K_X)|_F\geq
\frac{1}{m(2m_0+1)}M_{m(2m_0+1),X'}|_F\geq
\frac{1}{2m_0+1}\sigma^*(K_{F_0}). \eqno{(1)}$$

\begin{lem}\label{b=0} Keep the same notation as in \ref{setup}.
Assume $b=0$ and that a general fiber of $f$ has a Gorenstein
minimal model. Then there exists a sequence of positive rational
numbers $\{\beta_t\}$, with $\beta_t<\frac{p}{m_0+p}$ and
$\beta_t\underset{t\mapsto +\infty}\mapsto \frac{p}{m_0+p}$, such
that
$$\pi^*(K_X)|_F-\beta_t \sigma^*(K_{F_0})$$
is $\bQ$-linearly equivalent to an effective $\bQ$-divisor for all
integers $t>0$.
\end{lem}
\begin{proof} When $n=\dim(V)=3$, this is nothing but Theorem
\cite[Lemma 3.3]{Chen-Zuo}. Here we generalize it for any $n$. One
has $\mathcal {O}_{B}(p)\hookrightarrow {f}_*\omega_{X'}^{m_0}$ and
therefore ${f}_*\omega_{X'/B}^{t_0p}\hookrightarrow
{f}_*\omega_{X'}^{t_0p+2t_0m_0}$ for any positive integer $t_0$.

Note that ${f}_*\omega_{X'/B}^{t_0p}$ is generated by global
sections since it is semi-positive according to Viehweg \cite{V}. So
any local section can be extended to a global one. On the other
hand, whenever $t_0$ is bigger, $|t_0p\sigma^*(K_{F_0})|$ is free by
Base Point Free Theorem and is exactly the movable part of
$|t_0pK_F|$ by the ordinary projection formula. Clearly one has the
following relation:
$$a_0\pi^*(K_X)|_F\ge M_{t_0p+2t_0m_0,X'}|_F\geq b_0\sigma^*(K_{F_0})$$
where $a_0:=t_0p+2t_0m_0$ and $b_0:=t_0p$. This means that there is
an effective $\mathbb{Q}$-divisor $E_F'$ on $F$ such that
$$a_0\pi^*(K_X)|_F=_{\bQ} b_0\sigma^*(K_{F_0})+E_F'.$$
Thus $\pi^*(K_X)|_F =_{\bQ} \frac{p}{p+2m_0}\sigma^*(K_{F_0})+E_F$
with $E_F=\frac{1}{a_0}E_F'$.

Let us consider the case $p\ge 2$ first.

Assume we have defined $a_l$ and $b_l$ such that the following is
satisfied with $l = t:$
$$a_{l}\pi^*(K_X)|_F \ge b_{l}\sigma^*(K_{F_0}).$$
We will define $a_{t+1}$ and $b_{t+1}$ inductively such that the
above inequality is satisfied with $l = t+1$. One may assume from
the beginning (modulo necessary blow-ups) that the fractional part
of the support of $a_t\pi^*(K_X)$ is of simple normal crossing. Then
the Kawamata-Viehweg vanishing theorem gives the surjective map
$$H^0(K_{X'}+\roundup{a_t\pi^*(K_X)}+F)\longrightarrow H^0(F, K_F+
\roundup{a_t\pi^*(K_X)}|_F).$$ One has the relation
\begin{eqnarray*}
|K_{X'}+\roundup{a_t\pi^*(K_X)}+F||_F&=&|K_F+\roundup{a_t\pi^*(K_X)}|_F|\\
&\supset& |K_F+b_t\sigma^*(K_{F_0})|\\
&\supset& |(b_t+1)\sigma^*(K_{F_0})|.
\end{eqnarray*}
Denote by $M_{a_t+1,X'}'$ the movable part of $|(a_t+1)K_{X'}+F|$.
Applying \cite[Lemma 2.7]{MPCPS}, one has $M_{a_t+1,X'}'|_F\geq
(b_t+1)\sigma^*(K_{F_0}).$ Re-modifying our original $\pi$ such that
$|M_{a_t+1,X'}'|$ is base point free. In particular, $M_{a_t+1,X'}'$
is nef. Since $X$ is of general type $|mK_X|$ gives a birational map
whenever $m$ is big enough. Thus we see that $M_{a_t+1,X'}'$ is big
if we fix a very big $t_0$ in advance.

Now the Kawamata-Viehweg vanishing theorem again gives
\begin{eqnarray*}
|K_{X'}+M_{a_t+1,X'}'+F||_F&=&|K_F+M_{a_t+1,X'}'|_F|\\
&\supset& |K_F+(b_t+1)\sigma^*(K_{F_0})|\\
&\supset& |(b_t+2)\sigma^*(K_{F_0})|.
\end{eqnarray*}

Repeat the above procedure and denote by $M_{a_t+u,X'}'$ the movable
part of $|K_{X'}+M_{a_t+u-1,X'}'+F|$ for integers $u\ge 2$. For the
same reason, we may assume $|M_{a_t+u,X'}'|$ is base point free and
is thus nef and big. Inductively, for any $u>0$, one has:
$$M_{a_t+u,X'}'|_F\ge (b_t+u)\sigma^*(K_{F_0}).$$
Applying the vanishing theorem once more, one has
\begin{eqnarray*}
|K_{X'}+M_{a_t+u,X'}'+F||_F&=&|K_F+M_{a_t+u,X'}'|_F|\\
&\supset& |K_F+(b_t+u)\sigma^*(K_{F_0})|\\
&\supset& |(b_t+u+1)\sigma^*(K_{F_0})|.
\end{eqnarray*}

Take $u=p-1$. Noting that
$$|K_{X'}+M_{a_t+p-1,X'}'+F|\subset |(a_t+p+m_0)K_{X'}|$$
and applying \cite[Lemma 2.7]{MPCPS} again, one has
$$a_{t+1}\pi^*(K_X)|_F\ge M_{a_t+p+m_0,X'}|_F\geq
M'_{a_t+p,X'}|_F\geq b_{t+1} \sigma^*(K_{F_0})$$ where
$a_{t+1}:=a_t+p+m_0$ and $b_{t+1}=b_t+p$. Set $\beta_t =
\frac{b_{t}}{a_{t}}.$ Clearly $\lim_{t\mapsto +\infty} \beta_t =
\frac{p}{m_0+p}$.

The case $p=1$ can be considered similarly with a simpler
induction. We leave it as an exercise.
\end{proof}

\begin{lem}\label{diff} Let $\widetilde{\Lambda}\subset |L|$
be a pencil on $V$ where $L$ is a divisor on $V$. Let $R$ be a nef
and big $\bQ$-divisor on $V$. Assume there is a birational
modification $\pi_{\widetilde{\Lambda}}: V'\lra V$ such that:
\begin{itemize}
\item[(1)] the fractional part of $\pi_{\widetilde{\Lambda}}^*(R)$ has simple normal
crossing supports and the movable part of
$\pi_{\widetilde{\Lambda}}^*(\widetilde{\Lambda})$ is base point
free;

\item[(2)] $|K_{V'}+\roundup{\pi_{\widetilde{\Lambda}}^*(R)}|\neq \emptyset.$
\end{itemize}
Then $|K_V+\roundup{R}+L|$ distinguishes different irreducible
elements in the movable part of $\widetilde{\Lambda}$.
\end{lem}
\begin{proof} Noting that
$${\pi_{\widetilde{\Lambda}}}_*\OO_{V'}(K_{V'}+\pi_{\widetilde{\Lambda}}^*(\roundup{R})
+\pi_{\widetilde{\Lambda}}^*(L)|)\cong \OO_V(K_V+\roundup{R}+L)$$
and that $\pi_{\widetilde{\Lambda}}^*(\roundup{R})\geq
\roundup{\pi_{\widetilde{\Lambda}}^*(R)}$, we only need to study the
smaller linear system
$|K_{V'}+\roundup{\pi_{\widetilde{\Lambda}}^*(R)})+M_{\widetilde{\Lambda}}|$
where $|M_{\widetilde{\Lambda}}|$ is the movable part of
$\pi_{\widetilde{\Lambda}}^*(\widetilde{\Lambda})$. By our
assumption, $|M_{\widetilde{\Lambda}}|$ is composed with a pencil
and $f_{\widetilde{\Lambda}}: V'\lra B_{\widetilde{\Lambda}}$ is an
induced fibration of $\Phi_{|M_{\widetilde{\Lambda}}|}$, where
$B_{\widetilde{\Lambda}}$ is a smooth curve.

If $g(B_{\widetilde{\Lambda}})=0$, then
$|K_{V'}+\roundup{\pi_{\widetilde{\Lambda}}^*(R)}+M_{\widetilde{\Lambda}}|$
distinguishes different general fibers of $f_{\widetilde{\Lambda}}$
because $|K_{V'}+\roundup{\pi_{\widetilde{\Lambda}}^*(R)}|\neq
\emptyset.$

If $g(B_{\widetilde{\Lambda}})>0$, we pick up two general fibers
$F_{\widetilde{\Lambda}}'$ and $F_{\widetilde{\Lambda}}''$. Then 
$M_{\widetilde{\Lambda}}-F_{\widetilde{\Lambda}}'-F_{\widetilde{\Lambda}}''$ is nef and the
Kawamata-Viehweg vanishing theorem \cite{KV, VV} gives 
$$H^1(V',K_{V'}+\roundup{\pi_{\widetilde{\Lambda}}^*(R)}+M_{\widetilde{\Lambda}}- F_{\widetilde{\Lambda}}'-F_{\widetilde{\Lambda}}'')=0.$$
Thus follows the following surjective
map:
\begin{eqnarray*}
&&H^0(V',K_{V'}+\roundup{\pi_{\widetilde{\Lambda}}^*(R)}+M_{\widetilde{\Lambda}})\\
&\longrightarrow&
H^0(F_{\widetilde{\Lambda}}',(K_{V'}+\roundup{\pi_{\widetilde{\Lambda}}^*(R)})|_{F_{\widetilde{\Lambda}}'})
\oplus H^0(F_{{\widetilde{\Lambda}}''},(K_{V'}+
\roundup{\pi_{\widetilde{\Lambda}}^*(R)})|_{F_{\widetilde{\Lambda}}''}).
\end{eqnarray*}
 Again the assumption
$|K_{V'}+\roundup{\pi_{\widetilde{\Lambda}}^*(R)}|\neq \emptyset$
implies that
$|K_{V'}+\roundup{\pi_{\widetilde{\Lambda}}^*(R)}+M_{\widetilde{\Lambda}}|$
(and thus $|K_V+\roundup{R}+L|$) distinguishes different irreducible
elements in the movable part of $\widetilde{\Lambda}$. We are done.
\end{proof}

The following lemma is tacitly used in our context.

\begin{lem}\label{yes} Let $\bar{Q}$ be any
$\bQ$-divisor on a nonsingular projective variety $Z$. Let
$\widetilde{\pi}: \widetilde{Z}\lra Z$ be any birational
modification. Assume that
$|K_{\widetilde{Z}}+\roundup{\widetilde{\pi}^*(\bar{Q})}|$ gives a
birational map. Then $|K_Z+\roundup{\bar{Q}}|$ gives a birational
map.
\end{lem}
\begin{proof} This is clear due to the fact:
$$\widetilde{\pi}_*\OO_{\widetilde{Z}}(K_{\widetilde{Z}}+
\widetilde{\pi}^*(\roundup{\bar{Q}}))\cong
\OO_Z(K_Z+\roundup{\bar{Q}})$$ and
$\widetilde{\pi}^*(\roundup{\bar{Q}})\geq
\roundup{\widetilde{\pi}^*(\bar{Q})}$.
\end{proof}

\section{\bf $\bQ$-divisors on surfaces and threefolds}

We leave the proof for the next two results on surfaces as an exercise which is really a standard $\bQ$-divisor argument. 

\begin{lem}\label{pg1} Let $S$ be a nonsingular projective surface
of general type. Denote by $\sigma:S\lra S_0$ the birational
contraction onto the minimal model $S_0$. For any nef and big
${\mathbb Q}$-divisor $Q_2$ on $S$, one has
$$h^0(S, K_S+m\sigma^*(K_{S_0})+\roundup{Q_2})>1$$ under one of
the following situations:

(1) $m\geq 2$;

(2) $m=1$ and $p_g(S)>0$.
\end{lem}





\begin{thm}\label{surface} Keep the same notation as in Lemma
\ref{pg1}. Then the rational map
$\Phi_{|K_S+m\sigma^*(K_{S_0})+\roundup{Q_2}|}$ is birational in
either of the following cases:

(1) $m\ge 4$;

(2) $m\ge 3$ and $p_g(S)>0$.
\end{thm}

\begin{prop}\label{3-eff} Assume $\dim (V)=3$
and $P_{m_0}(V)\geq 2$ for some positive integer $m_0$.  Keep the
same notation as in \ref{setup}. Let $Q_3$ be a nef
 $\bQ$-divisor on $V$ and $Q'_3$ a nef $\bQ$-divisor on $X'$. Then
$K_{X'}+\roundup{q_3\pi^*(K_X)+Q'_3}+F$ is effective for all
rational numbers $q_3> 2m_0+2$. Consequently $mK_V+\roundup{Q_3}$ is
effective for all integers
 $m\geq 3m_0+4$.
\end{prop}

\begin{proof} The last statement is a direct application of the
first one due to \ref{reduction}. We prove the first statement.

\begin{quote}{\em
($\pounds$) Take further necessary modifications to $X'$ such that
the supports of the fractional parts of $Q_3'$ and $\pi_V^*(Q_3)$
are of simple normal crossing. For simplicity we still use $X'$ to
denote the final birational model dominating $V$.}
\end{quote}

We have a fibration $f:X'\lra B$ induced from $|m_0K|$ as in
\ref{setup}.  We consider the linear system
$$|K_{X'}+\roundup{q_3\pi^*(K_X)+Q_3'}+F|\subset
|(\roundup{q_3}+m_0+1)K_{X'}+\roundup{Q_3'}|.$$ Because
$q_3\pi^*(K_X)+Q_3'$ is nef and big and the fractional part of
$Q_3'$ is of simple normal crossing by our assumption, the vanishing
theorem says
$$H^1(X,K_{X'}+\roundup{q_3\pi^*(K_X)+Q_3'})=0.$$ Thus one has
the surjective map:
$$H^0(X,K_{X'}+\roundup{q_3\pi^*(K_X)+Q_3'}+F)\lra H^0(F,
K_F+\roundup{q_3\pi^*(K_X)+Q_3'}|_F).$$

Note that in our case a general fiber of $f$ is a surface of general
type which has a Gorenstein minimal model. Thus the conditions in
both Lemma \ref{b>0} and Lemma \ref{b=0} are satisfied. If $b>0$,
Lemma \ref{b>0} says that $\pi^*(K_X)|_F\sim \sigma^*(K_{F_0})$
where $\sigma: F\lra F_0$ is the contraction map. If $g(B)=0$, Lemma
\ref{b=0} says that one can find a very big number $s$ such that
$$\pi^*(K_X)|_F\geq \beta_s\sigma^*(K_{F_0})$$
and $\beta_s$ is sufficiently near $\frac{p}{m_0+p}\geq
\frac{1}{m_0+1}$ and $\beta_s<\frac{p}{m_0+p}$.

Let us put $\alpha_s:=\frac{p}{m_0+p}-\beta_s$. Then
$\alpha_s\mapsto 0$ whenever $s\mapsto +\infty$.

Whenever $q_3> 2m_0+2$, one has
\begin{eqnarray*} q_3\pi^*(K_X)|_F&\geq&
q_3\beta_s
\sigma^*(K_{F_0})=q_3(\frac{p}{m_0+p}-\alpha_s)\sigma^*(K_{F_0})\\
&=&(2+(\frac{q_3p-2m_0-2p}{m_0+p}-q_3\alpha_s))\sigma^*(K_{F_0}).
\end{eqnarray*}
When $s$ is big enough, one sees
$\frac{q_3p-2m_0-2p}{m_0+p}-q_3\alpha_s>0$. We may assume that
$q_3\pi^*(K_X)|_F-q_3\beta_s\sigma^*(K_{F_0})$ is $\bQ$-linearly
equivalent to an effective divisor $R_{q_3,s}$ on $F$. Then
\begin{eqnarray*}\widetilde{R_{q_3,s}}&:=&q_3\pi^*(K_X)|_F-R_{q_3,s}-2\sigma^*(K_{F,0})\\
& \equiv
&(\frac{q_3p-2m_0-2p}{m_0+p}-q_3\alpha_s)\sigma^*(K_{F_0})\end{eqnarray*}
is nef and big since $\frac{q_3p-2m_0-2p}{m_0+p}>0$. Therefore
$$H^0(K_F+\roundup{q_3\pi^*(K_X)+Q_3'}|_F)\supset
H^0(K_F+2\sigma^*(K_{F_0})+\roundup{\widetilde{R}_{q_3,n}+
Q_3'})\neq 0$$ by Lemma \ref{pg1}. And in fact
$h^0(K_F+2\sigma^*(K_{F_0})+\roundup{\widetilde{R}_{q_3,n}+
Q_3'})>1$.
\end{proof}

\begin{rem}\label{fiber} In the proof of Proposition \ref{3-eff},
if the general fiber of $f$ is a surface with $p_g>0$, then,
according to Lemma \ref{pg1},
$K_{X'}+\roundup{q_3\pi^*(K_X)+Q'_3)}+F$ is effective for all
rational numbers $q_3> m_0+1$. And accordingly $mK_V+\roundup{Q_3}$
is effective for all integers $m\geq 2m_0+3$.
\end{rem}

\begin{thm}\label{3folds} Assume $\dim (V)=3$
and $P_{m_0}(V)\geq 2$ for some positive integer $m_0$.  Keep the
same notation as in \ref{setup}. Let $Q_3$ be a nef $\bQ$-divisor on
$V$ and $Q_3'$ a nef $\bQ$-divisor on $X'$. Then
\begin{itemize}
\item[(1)] $\Phi_{|K_{X'}+\roundup{q_3\pi^*(K_X)+Q_3'}+M_0|}$ is
birational for all rational numbers $q_3>4m_0+4$. In particular
$\Phi_{|mK_V+\roundup{Q_3}|}$ is birational for all integers $m\geq
5m_0+6$;

\item[(2)] if the general fiber $F$ of $f\colon X'\lra B$ has positive geometric
genus, $\Phi_{|K_{X'}+\roundup{q_3\pi^*(K_X)+Q_3'}+M_0|}$ is
birational for all rational numbers $q_3> 3m_0+3$. In particular
$\Phi_{|mK_V+\roundup{Q_3}|}$ is birational for all integers $m\geq
4m_0+5$.
\end{itemize}
\end{thm}
\begin{proof} According to Proposition \ref{3-eff} and Remark \ref{fiber},
$$K_{X'}+\roundup{q_3\pi^*(K_X)+Q_3'}$$
is always effective under each situation since $m_0\pi^*(K_X)\geq
M_0$. Therefore Lemma \ref{diff} says that
$|K_{X'}+\roundup{q_3\pi^*(K_X)+Q_3'}+M_0|$ can distinguish
different generic irreducible elements of $|M_0|$. Thus it suffices
to prove the birationality of
$\Phi_{|K_{X'}+\roundup{q_3\pi^*(K_X)+Q_3'}+M_0|}|_F$ for a general
fiber $F$ of $f$. The proofs for statements (1) and (2) are similar.
We only consider (1) while omitting the proof for (2).

Of course, the first step in utilizing the vanishing theorem is to
make the support of the fractional part of $\{Q_3'\}$ to be simple
normal crossing. This can be done by re-modifying $X'$. For
simplicity we may assume, from now on, that our $X'$ has the
property stated in $(\pounds)$ (see the proof of Proposition
\ref{3-eff}).

The Kawatama-Viehweg vanishing theorem implies, noticing $F|_F\sim
0$, that
$$|K_{X'}+\roundup{q_3\pi^*(K_X)+Q_3'}+M_0||_F=|K_F+\roundup{q_3\pi^*(K_X)+Q_3'}|_F|.$$
We study a smaller system
$|K_F+\roundup{q_3\pi^*(K_X)|_F+{Q_3'}|_F}$. We have already
$\pi^*(K_X)|_F\geq \beta_s\sigma^*(K_{F_0})$ and
$0<\alpha_s:=\frac{p}{m_0+p}-\beta_s$, $\alpha_s\mapsto 0$ whenever
$s\mapsto +\infty$. When $q_3> 4m_0+4$,
\begin{eqnarray*}
q_3\pi^*(K_X)|_F&\geq & q_3\beta_s\sigma^*(K_F)\\
&= & 4\sigma^*(K_{F_0})+t_s\sigma^*(K_{F_0})
\end{eqnarray*}
where $t_s:=\frac{q_3p-4m_0-4p}{m_0+p}-q_3\alpha_s>0$ whenever $s$
is big enough. Therefore, by Theorem \ref{surface} and Lemma
\ref{yes},
$|K_S+\roundup{4\sigma^*(K_{F_0})+t_s\sigma^*(K_{F_0})+{Q_3'}}|_F|$
gives a birational map. Being a bigger linear system,
$$|K_F+\roundup{q_3\pi^*(K_X)+Q_3'}|_F|$$ also gives a birational map.
We are done.
\end{proof}

\begin{setup}\label{pg}{\bf Threefolds $V$ with $\chi(\OO_V)>1$.} Keep the
same notation as in \ref{setup}. As seen in \cite[Lemma
2.32]{Explicit_II}, if $\chi(\OO_V)>1$ and $q(V)=0$, then a general
fiber $F$ of $f:X'\lra B$ has the geometric genus $p_g(F)>0$.
\end{setup}

\begin{lem}\label{nu3} Assume $\dim (V)=3$. Then
$$\lambda(V)\leq \begin{cases}
18& \text{if}\ \chi(\OO_V)>1\ \text{and}\ q(V)=0;\\
10& otherwise.
\end{cases}$$
\end{lem}
\begin{proof} The first statement $\lambda(V)\leq 18$ is due to \cite[Theorem
4.8]{Explicit_II}.

When $q(V)>0$, $\lambda(V)\leq 3$ by Chen-Hacon \cite{C-H2}.

When $\chi(\OO_V)= 1$, $\lambda(V)\leq 10$ by \cite[Corollary
3.13]{Explicit_II}.

Finally when $\chi(\OO_V)<0$, $\lambda(V)\leq 3$ is a direct
consequence of Reid's plurigenus formula by Reid \cite{YPG} and by
Chen-Zuo  \cite[Lemma 4.1]{Chen-Zuo}.
\end{proof}

{}From now on, we classify a 3-fold $V$ into two types:{\em
\begin{itemize}
\item [(I)] $\chi(\OO_V)>1$ and $q(V)=0$;

\item [(II)] either $\chi(\OO_V)\leq 1$ or $q(V)>0$.
\end{itemize}}

\section{\bf Point separation on 4-folds}

\begin{prop}\label{4-eff} Assume $n=\dim (V)=4$ and $P_{m_0}(V)\geq 2$ for some
positive integer $m_0$. Keep the same notation as in \ref{setup}.
Let $Q_4$ be a nef $\bQ$-divisor on $V$ and $Q_4'$ a nef
$\bQ$-divisor on $X'$. Then
\begin{itemize}
\item[(1)] $K_{X'}+\roundup{q_4\pi^*(K_X)+Q_4'}+F$ is an effective divisor for all
rational numbers $q_4> 74m_0+37$;

\item[(2)] $mK_V+\roundup{Q}$ is effective for all integers $m\geq
75m_0+39.$
\end{itemize}
\end{prop}
\begin{proof} (2) is a direct result from (1) according to
\ref{reduction}. We only prove (1). Similar to assumption
($\pounds$) in the proof of Proposition \ref{3-eff}, we may assume
that $X'$ is good enough.

Because $q_4\pi^*(K_X)+Q_4'$ is nef and big, and its fractional
part has normal crossing supports by assumption, the
Kawamata-Viehweg vanishing theorem gives the surjective map:
$$H^0(X', K_{X'}+\roundup{q_4\pi^*(K_X)+Q_4'}+F)\lra H^0(F,
K_F+\roundup{q_4\pi^*(K_X)+Q_4'}|_F).$$ By Lemma \ref{b>0} and
inequality (1), one has
\begin{eqnarray*}
\roundup{q_4\pi^*(K_X)+Q_4'}|_F&\geq &
\roundup{(q_4\pi^*(K_X)+Q_4')|_F}\\
&\geq& \roundup{\frac{q_4}{2m_0+1}\sigma^*(K_{F_0})+{Q_4'}|_F}.
\end{eqnarray*}

If $F$ is of type (I), then, by Remark \ref{fiber}, \ref{pg} and
Lemma \ref{nu3}, we need to set $\frac{q_4}{2m_0+1}> 37\geq
(\lambda(F)+1)+\lambda(F)$, i.e. $q_4> 74m_0+37$, so that
$K_F+\roundup{\frac{q_4}{2m_0+1}\sigma^*(K_{F_0})+{Q_4'}|_F}$ is an
effective divisor on $F$.

If $F$ is of type (II), then by Proposition \ref{3-eff} and Lemma
\ref{nu3} we need $\frac{q_4}{2m_0+1}> 32\geq
(2\lambda(F)+2)+\lambda(F)$ so that
$K_F+\roundup{\frac{q_4}{2m_0+1}\sigma^*(K_{F_0})+{Q_4'}|_F}$ is
effective.

In a word, when $q_4> 74m_0+37$,
$K_{X'}+\roundup{q_4\pi^*(K_X)+Q_4'}+F$ is effective.
\end{proof}

\begin{thm}\label{4-folds} Assume $n=\dim (V)=4$ and $P_{m_0}\geq 2$ for some
positive integer $m_0$. Keep the same notation as in \ref{setup}.
Let $Q_4$ be a nef $\bQ$-divisor on $V$ and $Q_4'$ a nef
$\bQ$-divisor on $X'$.
 Then
\begin{itemize}
\item[(1)] $|K_{X'}+\roundup{q_4\pi^*(K_X)+Q_4'}+M_0|$ gives a
birational map for all rational numbers $q_4> 150m_0+75$;

\item[(2)] $\Phi_{|mK_V+\roundup{Q}|}$ is birational for all integers $m\geq
151m_0+77.$

\end{itemize}
\end{thm}
\begin{proof} Similar to assumption $(\pounds)$, we may assume
that $X'$ is good enough (after a necessary modification). Also (2)
is a direct result of (1). We only prove (1).

By Proposition \ref{4-eff}, we see that
$K_{X'}+\roundup{q_4\pi^*(K_X)+Q_4'}\geq 0$. Lemma \ref{diff} tells
us that we only need to verify the birationality of
$$\Phi_{|K_{X'}+\roundup{q_4\pi^*(K_X)+Q_4'}+M_0|}|_F$$ for a general
fiber $F$ of $f$. The vanishing theorem gives
$$|K_{X'}+\roundup{q_4\pi^*(K_X)+Q_4'}+M_0||_F=|K_F+\roundup{q_4\pi^*(K_X)+Q_4'}|_F|$$
noticing ${M_0}|_F\sim 0$. Lemma \ref{b>0} and inequality (1) imply
$$\roundup{q_4\pi^*(K_X)+Q_4'}|_F\geq
\roundup{\frac{q_4}{2m_0+1}\sigma^*(K_{F_0})+{Q_4'}|_F}.$$

Noting that $F$ is a threefold of general type, we still use a
similar argument to that in the proof of Proposition \ref{4-eff}.

If $F$ is of type (I), then, by Theorem \ref{3folds}(2), \ref{pg},
Lemma \ref{nu3} and Lemma \ref{yes},
$|K_F+\roundup{\frac{q_4}{2m_0+1}\sigma^*(K_{F_0})+{Q_4'}|_F}|$
gives a birational map when $$\frac{q_4}{2m_0+1}>75 \geq
(3\lambda(F)+3)+\lambda(F),$$ i.e. $q_4> 150m_0+75$.

If $F$ is of type (II), by Theorem \ref{3folds}(1) and Lemma
\ref{nu3},
$$\Phi_{|K_F+\roundup{\frac{q_4}{2m_0+1}\sigma^*(K_{F_0})+{Q_4'}|_F}|}$$
is birational when $\frac{q_4}{2m_0+1}> 54\geq
(4\lambda(F)+4)+\lambda(F)$, i.e. $q_4> 108m_0+54$.

To make a conclusion, $|K_{X'}+\roundup{q_4\pi^*(K_X)+Q_4'}+M_0|$
gives a birational map for all rational numbers $q_4> 150m_0+75$.
\end{proof}

A direct result of Theorem \ref{4-folds} is the following:

\begin{cor}\label{4} Assume $\dim (V)=4$ and $P_{m_0}\geq 2$
for some positive integer $m_0$. Then $\varphi_{m}$ is birational
onto its image for all integers $m\geq 151m_0+77$.
\end{cor}

\section{\bf Proof of the main theorem}

We organize the proof according to the value of $\iota$.
\bigskip

First we consider the case $\iota\geq n-2$.

\begin{prop}\label{d1} Assume $n=\dim(V)\geq 3$,
$P_{m_0}(V)\geq 2$ for some positive integer $m_0$ and $\iota\geq
n-2$. Keep the same notation as in \ref{setup}. Let $Q_n$ be any
nef $\bQ$-divisor on $V$ and $Q_n'$ any nef $\bQ$-divisor on $X'$.
Then
\begin{itemize}
\item[(1)] $K_{X'}+\roundup{q_n\pi^*(K_X)+Q_n'}+(n-3)M_0+F$ is effective for
all rational numbers $q_n>\text{min}\{2m_0+2, 22\}\cdot
(2m_0+1)^{n-3}$;

\item[(2)] $mK_V+\roundup{Q_n}$ is effective for all integers $m\geq
\text{min}\{2m_0+2, 22\}\cdot (2m_0+1)^{n-3}+m_0(n-2)+2$.
\end{itemize}
\end{prop}
\begin{proof} Noting that (2) is a direct application of (1), we
only prove (1). We are going to do an induction on $n$.

When $n=3$, Proposition \ref{3-eff} says that the statement is true
when $q_3>2m_0+2=(2m_0+2)\cdot(2m_0+1)^{n-3}+m_0(n-3)$. On the
other hand, we may replace $m_0$ with $\lambda(V)$. In fact, Lemma
\ref{nu3} gives $\lambda(V)\leq 18$ for type (I) and
$\lambda(V)\leq 10$ for type (II). Thus, by Proposition \ref{3-eff}
and Remark \ref{fiber}, $K_{X'}+\roundup{q_3\pi^*(K_X)+Q_3'}+F$ is
effective whenever $q_3>22$. Therefore statement (1) is true for
$q_3>\text{min}\{2m_0+2, 22\}\cdot (2m_0+1)^{n-3}+m_0(n-3)$.

Assume that statement (1) is correct for varieties of dimension
$\leq n-1$. Starting with a good model $X'$ satisfying \ref{aasum}, we can do the induction.
Noticing that $F$ is of
dimension $n-1$, we hope to reduce the problem onto $F$. Since
$\iota\geq n-2$, we have $\dim \varphi_{m_0}(F)\geq \dim
\Phi_{|M_0|}(F)\geq n-3=\dim (F)-2$ by the simple additivity
property. Because ${M_0}|_F\leq m_0{K_{X'}}|_F\sim m_0K_F$, we know
$\Phi_{|m_0K_F|}(F)\geq \Phi_{|{M_0}|_F|}(F)\geq \dim (F)-2$.
Because $q_n\pi^*(K_X)+Q_n'$ is nef and big and has simple normal
crossing fractional parts, the vanishing theorem gives the
surjective map:
\begin{eqnarray*}
&&H^0(X',K_{X'}+\roundup{q_n\pi^*(K_X)+Q_n'}+(n-3)M_0+F)\\
&\rightarrow& H^0(F, K_F+\roundup{q_n\pi^*(K_X)+Q_n'}|_F+(n-3){M_0}|_F)\\
&\supset& H^0(F, K_F+\roundup{q_n\pi^*(K_X)|_F+{Q_n'}|_F}+(n-3){M_0}|_F)\\
&\supset& H^0(F,
K_F+\roundup{q_n'\sigma^*(K_{F_0})+{Q_n'}|_F}+(n-4){M_0}|_F+{M_0}|_F)
\ \ \ \ (4)
\end{eqnarray*}
where $q_n'\geq \frac{q_n}{2m_0+1}>\text{min}\{2m_0+2, 22\}\cdot
(2m_0+1)^{(n-1)-3}$ by Lemma \ref{b>0} and inequality (1). Because
${M_0}|_F\leq m_0K_F$ and by taking those pencil $\Lambda_F\subset
|M_0||_F$, the induction and Lemma \ref{yes} tell us that
$K_F+\roundup{q_n'\sigma^*(K_{F_0})+{Q_n'}|_F}+((n-1)-3){M_0}|_F+{M_0}|_F$
is effective
 whenever $q_n'> \text{min}\{2m_0+2, 22\}\cdot
(2m_0+1)^{(n-1)-3}$. We are done.
\end{proof}

\begin{thm}\label{b1} Assume $n=\dim(V)\geq 3$,
$P_{m_0}(V)\geq 2$ for some positive integer $m_0$ and $\iota\geq
n-2$. Keep the same notation as in \ref{setup}. Let $Q_n$ be any
nef $\bQ$-divisor on $V$ and $Q_n'$ any nef $\bQ$-divisor on $X'$.
Then
\begin{itemize}
\item[(1)] $|K_{X'}+\roundup{q_n\pi^*(K_X)+Q_n'}+(n-2)M_0|$ gives a birational map for
all rational numbers $$q_n>\text{min}\{4m_0+4,
57\}\cdot(2m_0+1)^{n-3};$$

\item[(2)] $|mK_V+\roundup{Q_n}|$ gives a birational map for all
integers
$$m\geq  \text{min}\{4m_0+4,
57\}\cdot (2m_0+1)^{n-3}+m_0(n-2)+2.$$
\end{itemize}
\end{thm}
\begin{proof} Again since (2) is a direct application of (1), we
only prove (1). We are going to do an induction on $n$.

Under the assumption $q_n>\text{min}\{4m_0+4,
57\}\cdot(2m_0+1)^{n-3}$, since $m_0\pi^*(K_X)\geq M_0$, we see
\begin{eqnarray*}
&&K_{X'}+\roundup{q_n\pi^*(K_X)+Q_n'}+(n-3)M_0\\
&\geq&
K_{X'}+\roundup{\widetilde{q}_n\pi^*(K_X)+Q_n'}+(n-3)M_0+F\geq 0
\end{eqnarray*}
by Proposition \ref{d1} since
$\widetilde{q}_n:=q_n-m_0>\text{min}\{2m_0+2, 22\}\cdot
(2m_0+1)^{n-3}$. According to Lemma \ref{diff},
$|K_{X'}+\roundup{q_n\pi^*(K_X)+Q_n'}+(n-2)M_0|$ can distinguish
different fibers of $f$. We are left to show the birationality of
the rational map given by
$|K_{X'}+\roundup{q_n\pi^*(K_X)+Q_n'}+(n-2)M_0||_F$ for a general
fiber $F$ of $f$.

When $n=3$ and $q_3>4m_0+4=(4m_0+4)\cdot (2m_0+1)^{n-3}$, the
statement is nothing but Theorem \ref{3folds}(1). On the other hand,
we may replace $m_0$ with $\lambda(V)$. In fact, Lemma \ref{nu3}
gives $\lambda(V)\leq 18$ for type (I) and $\lambda(V)\leq 10$ for
type (II). Then, by Theorem \ref{3folds}(1) and (2),
$|K_{X'}+\roundup{q_n\pi^*(K_X)+Q_n'}+(n-2)M_0|$ gives a birational
map for all rational numbers $q_3>\text{max}\{57,44\}=57$. Therefore
statement (1) is true whenever $q_3>\text{min}\{4m_0+4, 57\}$.

Assume that statement (1) is correct for varieties of dimension
$\leq n-1$. Starting with a good model $X'$ satisfying \ref{aasum}, we can do the induction again. Still, we see
$\Phi_{|m_0K_F|}(F)\geq \Phi_{|{M_0}|_F|}(F)\geq \dim (F)-2$. We
hope to reduce to the problem on $F$. According to the relation
(4), we only need to study
$$|J_{n-1}|:=|K_F+\roundup{q_n'\sigma^*(K_{F_0})+{Q_n'}|_F}+(n-4){M_0}|_F+{M_0}|_F|$$
where $q_n'\geq \frac{q_n}{2m_0+1}>\text{min}\{4m_0+4, 57\}\cdot
(2m_0+1)^{(n-1)-3}$ by Lemma \ref{b>0} and inequality (1). Because
${M_0}|_F\leq m_0K_F$ and by taking those pencil $\Lambda_F\subset
|M_0||_F$, the induction and Lemma \ref{yes} tell us that
$|K_F+\roundup{q_n'\sigma^*(K_{F_0})+{Q_n'}|_F}+((n-1)-2){M_0}|_F|$
gives a birational map. Thus
$|K_{X'}+\roundup{q_n\pi^*(K_X)+Q_n'}+(n-2)M_0||_F$ gives a
birational map. We are done.
\end{proof}

Next we discuss the case $\iota=n-3$.

\begin{prop}\label{d2} Assume $n=\dim(V)\geq 4$,
$P_{m_0}(V)\geq 2$ for some positive integer $m_0$ and $\iota\geq
n-3$. Keep the same notation as in \ref{setup}. Let $Q_n$ be any
nef $\bQ$-divisor on $V$ and $Q_n'$ any nef $\bQ$-divisor on $X'$.
Then
\begin{itemize}
\item[(1)] $K_{X'}+\roundup{q_n\pi^*(K_X)+Q_n'}+(n-4)M_0+F$ is effective for
all rational numbers $q_n>37(2m_0+1)^{n-3}$;

\item[(2)] $mK_V+\roundup{Q_n}$ is effective for all integers $m\geq
37(2m_0+1)^{n-3}+m_0(n-3)+2$.
\end{itemize}
\end{prop}
\begin{proof} Statement (1) implies (2). So we only prove (1). Similar to the assumption $(\pounds)$, we
may assume that $X'$ is good enough (modulo blow-ups) for our
purpose.  We prove by an induction on $n$.

When $n=4$, (1) is exactly Proposition \ref{4-eff}(1).

Assume that (1) is correct for all varieties of dimension $n-1$.
Pick a general fiber $F$ of $f$. Because ${M_0}|_F\leq
m_0{K_{X'}}|_F\sim m_0K_F$, we know $\Phi_{|m_0K_F|}(F)\geq
\Phi_{|{M_0}|_F|}(F)\geq n-4=\dim (F)-3$. As long as we take those
pencils $\Lambda_F\subset |M_0||_F$ on $F$, the induction works on
$F$. Thus we restrict everything onto $F$. By the vanishing
theorem, we may get the similar relation to (4):
\begin{eqnarray*}
&&|K_{X'}+\roundup{q_n\pi^*(K_X)+Q_n'}+(n-4)M_0+F||_F\\
&\supset&
|K_F+\roundup{q_n'\sigma^*(K_{F_0})+{Q_n'}|_F}+((n-1)-4){M_0}|_F+{M_0}|_F)|
\end{eqnarray*}
where $q_n'\geq \frac{q_n}{2m_0+1}>37(2m_0+1)^{(n-1)-3}$ by Lemma
\ref{b>0} and inequality (1). The later linear system is non-empty
by induction. Therefore
$K_{X'}+\roundup{q_n\pi^*(K_X)+Q_n'}+(n-4)M_0+F$ is effective for
all rational numbers $q_n>37(2m_0+1)^{n-3}$.
\end{proof}

\begin{thm}\label{b2} Assume $n=\dim(V)\geq 4$,
$P_{m_0}(V)\geq 2$ for some positive integer $m_0$ and $\iota\geq
n-3$. Keep the same notation as in \ref{setup}. Let $Q_n$ be any
nef $\bQ$-divisor on $V$ and $Q_n'$ any nef $\bQ$-divisor on $X'$.
Then
\begin{itemize}
\item[(1)] $K_{X'}+\roundup{q_n\pi^*(K_X)+Q_n'}+(n-4)M_0+F$ is effective for
all rational numbers $q_n>75(2m_0+1)^{n-3}$;

\item[(2)] $mK_V+\roundup{Q_n}$ is effective for all integers $m\geq
75(2m_0+1)^{n-3}+m_0(n-3)+2$.
\end{itemize}
\end{thm}
\begin{proof} The proof is parallel to that of Proposition \ref{d2}.
To avoid unnecessary redundancy, we omit the details.
\end{proof}

\begin{defn} The sequences $\{u_t\}_{t=4}^n$ and $\{w_t\}_{t=4}^n$ are defined by the following rules:
\begin{itemize}
\item $\widetilde{\lambda}_n=m_0$ and,
for all $i<n$, $\widetilde{\lambda}_i=\lambda_i$;

\item $u_4=75\widetilde{\lambda}_4+37$ and  $w_4=151\widetilde{\lambda}_4+75$;

\item for all $i$, $u_i=\widetilde{\lambda}_i+u_{i-1}(2\widetilde{\lambda}_i+1)$
and $w_i=\widetilde{\lambda}_i+w_{i-1}(2\widetilde{\lambda}_i+1)$.
\end{itemize}
\end{defn}

{}Finally we study the case $\iota\leq n-4$. We begin with the case
$\iota=1$.

\begin{thm}\label{d31} Assume $n=\dim (V)\geq 4$ and
$P_{m_0}\geq 2$ for some positive integer $m_0$ and $\iota\geq 1$.
Keep the same notation as in \ref{setup}. Let $Q_n$ be any nef
$\bQ$-divisor on $V$ and $Q_n'$ any nef $\bQ$-divisor on $X'$. Then
\begin{itemize}
\item[(1)] $K_{X'}+\roundup{\widetilde{q}_n\pi^*(K_X)+Q_n'}$ is effective for
all rational numbers $\widetilde{q}_n>u_n$.

\item[(2)] $|K_{X'}+\roundup{\widetilde{q}_n\pi^*(K_X)+Q_n'}|$ gives a
birational map for all rational numbers $\widetilde{q}_n>w_n$. 

\item[(3)] $|mK_V+\roundup{Q_n}|$ gives a birational map for all integers $m\geq
w_n+2.$
\end{itemize}
\end{thm}
\begin{proof} Statement (3) is a direct result of (2). So we have to prove (1) and (2).

When $n=4$, the conditions in (1) and (2) read
$\widetilde{q}_4>u_4:=75m_0+37$ and
$\widetilde{q}_4>w_4:=151m_0+75$. Both the statements are nothing
but Proposition \ref{4-eff}(1) and Theorem \ref{4-folds}(1), noting
that $m_0\pi^*(K_X)\geq M_0\geq F$. Besides if we take
$m_0=\lambda(V)$, the statements are true for
$\widetilde{q}_4>u_4:=75\lambda(V)+37$ and
$\widetilde{q}_4>w_4:=151\lambda(V)+75$.

Assume that the statements are correct for $n-1$ dimensional
varieties. By definition, $m_0\geq \lambda(V)$ and $\lambda_n\geq
\lambda(V)$. Because $\widetilde{q}_n\pi^*(K_X)\geq
(\widetilde{q}_n-m_0)\pi^*(K_X)+M_0$, we will study
$|K_{X'}+\roundup{(\widetilde{q}_n-m_0)\pi^*(K_X)+Q_n'}+M_0|$. Now
the vanishing theorem gives:
\begin{eqnarray*}
&&|K_{X'}+\roundup{(\widetilde{q}_n-m_0)\pi^*(K_X)+Q_n'}+M_0||_F\\
&=& |K_F+\roundup{(\widetilde{q}_n-m_0)\pi^*(K_X)+Q_n'}|_F|\\
&\supset&
|K_F+\roundup{(\widetilde{q}_n-m_0)\pi^*(K_X)|_F+{Q_n'}|_F}\\
&\supset&
|K_F+\roundup{\frac{\widetilde{q}_n-m_0}{2m_0+1}\sigma^*(K_{F_0})+{Q_n'}|_F}|.
\end{eqnarray*}
Clearly $\dim (F)=n-1$, the induction hypothesis says
$$K_F+\roundup{\frac{\widetilde{q}_n-m_0}{2m_0+1}\sigma^*(K_{F_0})+{Q_n'}|_F}$$
is effective when $\frac{\widetilde{q}_n-m_0}{2m_0+1}>u_{n-1}$ and
$|K_F+\roundup{\frac{\widetilde{q}_n-m_0}{2m_0+1}\sigma^*(K_{F_0})+{Q_n'}|_F}|$
gives a birational map when
$\frac{\widetilde{q}_n-m_0}{2m_0+1}>w_{n-1}$. Both conditions can
be replaced by
$\widetilde{q}_n>u_n=\widetilde{\lambda}_n+u_{n-1}(2\widetilde{\lambda}_n+1)$
and
$\widetilde{q}_n>w_n=\widetilde{\lambda}_n+w_{n-1}(2\widetilde{\lambda}_n+1)$,
where $\widetilde{\lambda}_n=m_0$. Note however it is enough to
take $\widetilde{\lambda}_i=\lambda_i$ for all $i<n$. We are done.
\end{proof}

\begin{thm}\label{b3} Assume $n=\dim (V)\geq 5$ and
$P_{m_0}\geq 2$ for some positive integer $m_0$ and $\iota\leq
n-4$. Keep the same notation as in \ref{setup}. Let $Q_n$ be any
nef $\bQ$-divisor on $V$ and $Q_n'$ any nef $\bQ$-divisor on $X'$.
Then
\begin{itemize}
\item[(1)] $K_{X'}+\roundup{q_n\pi^*(K_X)+Q_n'}+(\iota-1)M_0$ is effective for
all rational numbers $q_n>(2m_0+1)^{\iota-1} u_{n-\iota+1}$. 

\item[(2)] $|K_{X'}+\roundup{q_n\pi^*(K_X)+Q_n'}+(\iota-1)M_0|$ gives a
birational map for all rational numbers $q_n>(2m_0+1)^{\iota-1}
w_{n-\iota+1}$. 

\item[(3)] $|mK_V+\roundup{Q_n}|$ gives a birational map for all integers $m\geq
(2m_0+1)^{\iota-1}w_{n-\iota+1}+m_0(\iota-1)+2.$
\end{itemize}
\end{thm}
\begin{proof} Statement (3) is direct from (2). So we only need to prove (1)
and (2).

When $n=5$, one necessarily has $\iota=1$ and the statements are
nothing but those in Theorem \ref{d31}.

First We consider statement (1). We may restrict the problem to $F$
by the vanishing theorem. Then, since $M_{0}|_F\leq
m_0\sigma^*(K_{F_0})\leq m_0K_F$, we may study the linear system
$$|K_F+\roundup{\frac{q_n}{2m_0+1}\sigma^*(K_{F_0})+{Q_n'}|_F}+(\iota-2){M_0}|_F|. \eqno{(5)}$$
Note that $\dim \Phi_{|m_0K_F|}(F)\geq \Phi_{|M_0|}(F)\geq \iota-1$.
Then we can do an induction and repeat this program for finite
times. Finally we are reduced to study the non-emptyness of the
linear system on $W$ of dimension $n-\iota+1$:
$$|K_W+\roundup{\frac{q_n}{(2m_0+1)^{\iota-1}}\tau^*(K_{W_0})+{Q_n'}|_W}| \eqno{(6)}$$
where $\tau\colon W\rightarrow W_0$ is a contraction morphism to
the minimal model. Furthermore $\dim \Phi_{|m_0K_W|}(W)\geq \dim
\Phi_{|M_0|}(W)\geq 1$. Now Theorem \ref{d31}(1) says that
$\frac{q_n}{(2m_0+1)^{\iota-1}}>u_{n-\iota+1}$ is enough to secure
the non-emptyness of the linear system (6), where $u_{n-\iota+1}$ is
obtained by the sequence $\{u_t\}_{t=4}^{n-\iota+1}$ with
$u_i=\widetilde{\lambda}_i+u_{i-1}(2\widetilde{\lambda}_i+1)$,
$u_4=75\widetilde{\lambda}_4+37$,
$\widetilde{\lambda}_{n-\iota+1}=m_0$ and, for all other $i$,
$\widetilde{\lambda}_i=\lambda_i$. Therefore statement (1) is
correct.

Statement (1) and Lemma \ref{diff} allow us to reduce the problem
onto lower dimensional varieties. Thus what we are left to do is
similar to that for statement (1). So after one step restriction, we
get the linear system (5) on $F$. After successive restrictions and
inductions, we may obtain the linear system (6) on $W$ of dimension
$n-\iota+1$. Now we may apply Theorem \ref{d31}(2) to get the
condition $\frac{q_n}{(2m_0+1)^{\iota-1}}>w_{n-\iota+1}$ where
$w_{n-\iota+1}$ is obtained by the sequence
$\{w_t\}_{t=4}^{n-\iota+1}$ with
$w_i=\widetilde{\lambda}_i+w_{i-1}(2\widetilde{\lambda}_i+1)$,
$w_4=151\widetilde{\lambda}_4+75$,
$\widetilde{\lambda}_{n-\iota+1}=m_0$ and, for all other $i$,
$\widetilde{\lambda}_i=\lambda_i$. We are done.
\end{proof}

Finally we propose the following:

\begin{OP} As we have seen, inequality (1) in Section 2 is
the key step to get optimal birationality. Can one find a better
constant $\gamma>\frac{1}{2m_0+1}$ such that $\pi^*(K_X)|_F\geq
\gamma \sigma^*(K_{F_0})$? When $\dim(V)=3$,
$\gamma=\frac{p}{m_0+p}$ is nearly optimal by virtue of our
previous work. \end{OP}

\begin{setup}{\bf Acknowledgment.} This note grew out of
discussions with Jun Li to whom I feel considerably indebted. I
would like to thank both Jun Li and the Mathematics Research Center
of Stanford University for the support of my visit in the Spring of
2007. Thanks are also due to Jungkai A. Chen, Christopher D. Hacon, Yujiro Kawamata, Eckart Viehweg, De-Qi Zhang and Kang Zuo for their generous helps and stimulating
discussions. Finally I am grateful to the referee for several technical  suggestions.
\end{setup}


\end{document}